\documentclass[11pt]{amsart}
\usepackage{amsmath,amsthm,amsfonts,amssymb,latexsym,epsfig}

\newtheorem{theorem}{Theorem}[section]

\newtheorem{lemma}{Lemma}[section]

\newtheorem{cor}{Corollary}[section]

\numberwithin{equation}{section}

\theoremstyle{definition}

\theoremstyle{remark}

\begin{document}
\title{On lower and upper bounds of matrices}
\author{Peng Gao}
\address{Division of Mathematical Sciences, School of Physical and Mathematical Sciences,
Nanyang Technological University, 637371 Singapore}
\email{penggao@ntu.edu.sg}
%%\date{\today}
\date{June 16, 2009.}
\subjclass[2000]{Primary 47A30} \keywords{Lower and upper bounds
of matrices}

%%-------------------------------------------------------------------------------------------

\begin{abstract}
  Using an approach of Bergh, we give an alternate proof of Bennett's result on lower bounds for
non-negative matrices acting on non-increasing non-negative
sequences in $l^p$ when $p \geq 1$ and its dual version, the upper
bounds when $0<p \leq 1$. We also determine such bounds explicitly
for some families of matrices.
\end{abstract}

\maketitle
%%-------------------------------------------------------------------------------------------
\section{Introduction}
\label{sec 1} \setcounter{equation}{0}
%%-------------------------------------------------------------------------------------------

  Let $p>0$ and $l^p$ be the space of all complex sequences ${\bf a}=(a_n)_{n \geq 1}$ satisfying:
\begin{equation*}
  ||{\bf a}||_p=\Big(\sum^{\infty}_{i=1}|a_i|^p \Big )^{1/p}<\infty.
\end{equation*}
  When $p>1$, the celebrated
   Hardy's inequality \cite[Theorem 326]{HLP} asserts that for
   any ${\bf a} \in l^p$,
\begin{equation}
\label{eq:1} \sum^{\infty}_{n=1}\big{|}\frac {1}{n}
\sum^n_{k=1}a_k\big{|}^p \leq \Big(\frac {p}{p-1}
\Big)^p\sum^\infty_{k=1}|a_k|^p.
\end{equation}

   Hardy's inequality can be interpreted as the $l^p$ operator norm of
   the Ces\`{a}ro matrix $C$, given by
$c_{j,k} = 1/j, k \leq j$, and $0$ otherwise, is bounded on $l^p$
and has norm $\leq p/(p -1)$ (The norm is in fact $p/(p - 1)$). It
is known that the Ces\`{a}ro operator is not bounded below, or the
converse of inequality \eqref{eq:1} does not hold for any positive
constant. However, if one assumes $C$ acting only on
non-increasing non-negative sequences in $l^p$, then such a lower
bound does exist, and this is first obtained by Lyons in \cite{L}
for the case of $l^2$ with the best possible constant. For the
general case concerning the lower bounds for an arbitrary
non-negative matrix acting on non-increasing non-negative
sequences in $l^p$ when $p \geq 1$, Bennett \cite{B} determined
the best possible constant. When $0<p \leq 1$, one can also
consider a dual question and this has been studied in \cite{B1},
\cite{Bergh} and \cite{BGE}. Let $A=(a_{j,k}), 1 \leq j \leq m, 1
\leq k \leq n$ with $a_{j,k} \geq 0$, we can summarize the main
results in this area in the following
\begin{theorem}[{\cite[Theorem 2]{B}, \cite[Theorem 4]{BGE}}]
\label{thm1}
  Let ${\bf x}=(x_1, \ldots, x_n), x_1 \geq \ldots \geq x_n \geq 0$, $p \geq 1$,
  $0<q \leq p$, then
\begin{equation}
\label{1}
  ||A{\bf x}||_q \geq \lambda ||{\bf x}||_p,
\end{equation}
   where
\begin{equation*}
   ||A{\bf
   x}||^q_q=\sum^m_{j=1}(\sum^n_{k=1}a_{j,k}x_k)^q
\end{equation*}
   and
\begin{equation}
\label{3}
   \lambda^q =\min_{1 \leq r \leq n}
   r^{-q/p}\sum^m_{j=1}(\sum^r_{k=1}a_{j,k})^q.
\end{equation}
   Inequality \eqref{1} is reversed when $0<p \leq 1$ and $q \geq p$ with $\min$ replaced by $\max$ in \eqref{3}. Moreover,
   there is equality in \eqref{1} if ${\bf x}$ has the form
  $x_k=x_1, 1 \leq k \leq s$ and $x_k=0, k >s$ where $s$ is any
  value of $r$ where the minimum or maximum in \eqref{3} occurs.
\end{theorem}

   One may also consider the integral analogues of Theorem \ref{thm1} and there is a rich literature on this
   area
   and we shall refer the reader to the
   articles \cite{Bergh}, \cite{N}, \cite{Bu}, \cite{S}, \cite{Lai},
    \cite{BBP}, \cite{BBP1}, \cite{BPS}, \cite{CRS} and
   the reference therein for the related studies. We point out
   here one may deduce Theorem \ref{thm1} from its integral
   analogues by considering suitable integrals on suitable measure spaces (see for
   example, \cite{BPS} and \cite{CRS}).

    A special case of Theorem \ref{thm1} appeared in
   \cite{B1}, where Bennett established the following inequality
   for $0<p<1, x_1 \geq x_2 \geq \ldots \geq 0$,
\begin{equation}
\label{6}
   \sum^{\infty}_{n=1}\Big (\frac 1{n} \sum^{\infty}_{k=n}x_k \Big
   )^p \leq \frac {\pi p}{\sin \pi p}\sum^{\infty}_{n=1}x^p_n.
\end{equation}
   The constant $\pi p/\sin (\pi p)$ is best possible. An integral analogue of the above inequality was established by
   Bergh in \cite{Bergh} and he then used it to deduce a slightly
   weaker result than inequality \eqref{6}.

   Our interest in Theorem \ref{thm1} starts from the
   following inequality ($0<p<1$) for any non-negative ${\bf x}$:
\begin{equation}
\label{17}
  \sum^{\infty}_{n=1}\Big(\frac 1{n} \sum^{\infty}_{k=n}x_k \Big
  )^p \geq c_p \sum^{\infty}_{n=1}x^p_n.
\end{equation}
   It is shown in \cite[Theorem 345]{HLP} that the above inequality holds with $c_p=p^p$ for $0<p<1$ and it is also noted there that the constant $p^p$ may
   not be best possible and the best possible constant was in fact later obtained by Levin and Ste\v ckin
  \cite[Theorem 61]{L&S} to be $(p/(1-p))^p$ for $0<p \leq 1/3$.
  Recently, the author \cite{G9} has extended the result
  of Levin and Ste\v ckin to hold for
  $0 < p  \leq 0.346$. Inequalities of type \eqref{17} with more general weights are also studied in \cite{G9}, among which the
  following one for $0<p<1, \alpha \geq 1$:
\begin{equation*}
  \sum^{\infty}_{n=1}\Big(\frac 1{n^{\alpha}} \sum^{\infty}_{k=n}((k+1)^{\alpha}-k^{\alpha})x_k \Big
  )^p \geq c_{p, \alpha} \sum^{\infty}_{n=1}x^p_n.
\end{equation*}
   Here $c_{p, \alpha}$ is a constant and note that the above inequality gives back \eqref{17} when $\alpha=1$.
   In view of \eqref{6}, it's then natural to consider the reversed inequality if
   we assume further that $x_1 \geq x_2 \geq \ldots \geq 0$.

   It is our goal in this
   paper to first give an alternate proof of Theorem \ref{thm1}
   in Section \ref{sec 2} using the approach of Bergh in \cite{Bergh} and then
   using Theorem \ref{thm1} to prove the following result in
   Section \ref{sec 3}:
\begin{theorem}
\label{thm1.2}
  Let $0<p < 1, \alpha \geq 1, \alpha p<1, 0 \leq t \leq 1, x_1 \geq x_2 \geq \ldots \geq 0$. We
  have
\begin{equation}
\label{8}
   \sum^{\infty}_{n=1}\Big(\frac 1{n^{\alpha}} \sum^{\infty}_{k=n}((k+t)^{\alpha}-(k+t-1)^{\alpha})x_k \Big
  )^p \leq \frac {1}{\alpha}B(\frac {1}{\alpha}-p,p+1)
  \sum^{\infty}_{n=1}x^p_n,
\end{equation}
   where $B(x,y), x>0, y>0$ is the beta function
\begin{equation*}
   B(x,y)=\int^1_0t^{x-1}(1-t)^{y-1}dt.
\end{equation*}
   Inequality \eqref{8} also holds for $t=1, 0< \alpha <1$, $(1-\alpha)/(1+\alpha^2) \leq p \leq 1$. Moreover, the constant is best
   possible. Inequality \eqref{8} reverses when $t=1, \alpha>0, p \geq 1$
   with the best possible constant $(2^{\alpha}-1)^p$.
\end{theorem}

  We note that the case $\alpha=1, t=1, 0<p<1$ in the above theorem gives
  back inequality \eqref{6} and the integral analogue of Theorem
  \ref{thm1.2} has been studied in \cite{Lai}. Some consequences of Theorem \ref{thm1.2} are deduced in Section
  \ref{sec 4} and other applications of Theorem \ref{thm1} are
  given in Sections \ref{sec 5} and \ref{sec 6}.

%%-------------------------------------------------------------------------------------
\section{Proof of Theorem \ref{thm1}}
\label{sec 2} \setcounter{equation}{0}
%%-------------------------------------------------------------------------------------

   We need a lemma first:
\begin{lemma}
\label{lem1}
   Let $p \geq 1, 0<q \leq p$, then for any positive sequences $(a_j)_{1 \leq j \leq m}$ and any non-negative sequence $(b_j)_{1 \leq j \leq m}$, we have
\begin{equation}
\label{7}
   \Big (\sum^m_{j=1}(a_j+b_j)^q\Big )^{p/q-1}\Big (\sum^m_{j=1}(a_j+b_j)^{q-1}a_j\Big
   ) \geq \Big (\sum^m_{j=1}a^q_j\Big )^{p/q}.
\end{equation}
   The above inequality reverses when $0< p \leq 1,  q \geq p$.
\end{lemma}
\begin{proof}
     We shall only consider the case $p \geq 1, 0<q \leq p$ here, the other case is being analogue. We recast inequality \eqref{7} as
\begin{equation*}
   \Big (\sum^m_{j=1}(a_j+b_j)^q\Big )^{1-q/p}\Big (\sum^m_{j=1}(a_j+b_j)^{q-1}a_j\Big
   )^{q/p} \geq \sum^m_{j=1}a^q_j.
\end{equation*}
   Applying H\"older's inequality to the left-hand side expression
   above, we obtain
\begin{eqnarray*}
   \Big (\sum^m_{j=1}(a_j+b_j)^q\Big )^{1-q/p}\Big (\sum^m_{j=1}(a_j+b_j)^{q-1}a_j\Big
   )^{q/p} \geq \sum^m_{j=1}(a_j+b_j)^{q(1-1/p)}a^{q/p}_j \geq \sum^m_{j=1}a^q_j.
\end{eqnarray*}
   This completes the proof.
\end{proof}

   We now prove Theorem \ref{thm1}. As the proofs are similar for both cases, we shall focus only on establishing
   \eqref{1} for $p \geq 1, 0<q \leq p$ and we shall also leave the discussion on the cases of equality
   to the reader. We may also assume $a_{j,k} >0$ for all $j, k$ and the general case follows from a limiting process. By
   homogeneity, we see that one can make inequality \eqref{1} valid by taking $\lambda$ to be $\lambda_0=\min \{ ||A{\bf x}||_q : ||x||_p=1, x_1 \geq \ldots \geq x_n \geq 0
   \}$. By compactness, $\lambda_0$ is attained at some ${\bf x}_0 \neq 0$. We may assume the right-hand side expression of \eqref{3} is $>0$ for otherwise
   inequality \eqref{1} holds trivially. This readily implies that
   $\lambda_0 \neq 0$. Certainly $\lambda_0$ is no more than the right-hand expression of \eqref{3} and suppose now that $\lambda_0$ is strictly less than
   the right-hand expression of \eqref{3} and it's attained at a
   vector ${\bf x}_0$ satisfying: $({\bf x}_0)_k=x, 1 \leq k \leq i$ for some $k$ with $1 \leq i \leq
   n-1$ and $({\bf x}_0)_{i+1}=1<x$ (by homogeneity). We now regard
   $x$ as a variable and consider the following function:
\begin{equation*}
   f(x)=\frac {||A{\bf
   x}_0||^p_q}{||{\bf x}_0||^p_p}.
\end{equation*}
  We then have at ${\bf x}_0$,
\begin{align*}
   f'(x)  = \frac {p}{||{\bf x}_0||^p_p}\Big(x^{-1}||A{\bf
   x}_0||^{p-q}_q\sum^m_{j=1}(\sum^n_{k=1}a_{j,k}({\bf x}_0)_k)^{q-1}\sum^i_{k=1}a_{j,k}({\bf x}_0)_k-ix^{p-1}\frac {||A{\bf
   x}_0||^p_q}{||{\bf x}_0||^p_p} \Big ).
\end{align*}
   We set $a_j=\sum^i_{k=1}a_{j,k}({\bf x}_0)_k$ and $b_j=\sum^n_{k=i+1}a_{j,k}({\bf
   x}_0)_k$ (note that $a_j>0$) in Lemma \ref{lem1} to see that
\begin{align*}
   ||A{\bf
   x}_0||^{p-q}_q\sum^m_{j=1}(\sum^n_{k=1}a_{j,k}({\bf x}_0)_k)^{q-1}\sum^i_{k=1}a_{j,k}({\bf x}_0)_k
   \geq  \Big(\sum^m_{j=1}(\sum^i_{k=1}a_{j,k}x)^{q} \Big )^{p/q}.
\end{align*}
   It follows that
\begin{equation*}
   f'(x)  \geq \frac {p}{||{\bf x}_0||^p_p}\Big(x^{-1}\Big(\sum^m_{j=1}(\sum^i_{k=1}a_{j,k}x)^{q}\Big )^{p/q}-ix^{p-1}\lambda^p_0 \Big
   )>0.
\end{equation*}
   This leads to a contradiction and Theorem \ref{thm1} is thus
   proved.

%%-------------------------------------------------------------------------------------
\section{Proof of Theorem \ref{thm1.2}}
\label{sec 3} \setcounter{equation}{0}
%%-------------------------------------------------------------------------------------

   We first consider the case when $0<p<1$. We may certainly focus on establishing our assertion for
   inequality \eqref{8} with the infinite sums there replaced by any
   finite sums, say from $1$ to $N$. We now consider the case $\alpha \geq 1, t=1$. Theorem \ref{thm1} readily
   implies that in this case, the best constant is given by $\max_{1 \leq r \leq
   N} s_r$, where
\begin{equation*}
   s_r=
   r^{-1}\sum^{r}_{k=1}\Big(\frac {(r+1)^{\alpha}-k^{\alpha}}{k^{\alpha}} \Big
  )^p.
\end{equation*}
   Suppose we can show that the sequence $(s_r)$ is non-decreasing, then the maximum occurs when $r=N$, and as
   $N \rightarrow \infty$, one obtains the constant in \eqref{8}
   easily and this also shows that the constant there is best
   possible. It rests thus to show the sequence $(s_r)$ is
   non-decreasing. To show this, we use the trick of Bennett in
   \cite{B1} (see also \cite[Proposition 7]{B&J}) on considering the following function:
\begin{equation*}
   f_{\alpha, p}(x)=\Big (\frac {1-x^{\alpha}}{x^{\alpha}}\Big )^p+\Big (\frac {1-(1-x)^{\alpha}}{(1-x)^{\alpha}}\Big
   )^p.
\end{equation*}
   For $n \geq 1$ and any given function $f$ defined on $(0, 1)$,
   we define
\begin{equation*}
    A_n(f)=\frac {1}{n}\sum^{n}_{r=1}f\Big(\frac {r}{n+1} \Big).
\end{equation*}
   Note that we then have $s_n=2A_{n}(f_{\alpha, p})$. It then suffices to show that $A_{n}(f_{\alpha, p})$ increases with $n$.
   A result of Bennett and Jameson \cite[Theorem 1]{B&J}
    asserts that if $f$ is a convex function on $(0, 1)$, then $A_n(f)$ increases with
    $n$. Thus, it suffices to show that $f_{\alpha, p}$ is convex
    on $(0,1)$ and direct calculation shows that
\begin{align*}
   f''_{\alpha, p}(x) & = \alpha p x^{-\alpha p
   -2}(1-x^{\alpha})^{p-2}(\alpha p+1-(1+\alpha)x^{\alpha}) \\
   &+\alpha p (1-x)^{-\alpha p
   -2}(1-(1-x)^{\alpha})^{p-2}(\alpha
   p+1-(1+\alpha)(1-x)^{\alpha}).
\end{align*}
   As $f''_{\alpha, p}(x)=f''_{\alpha, p}(1-x)$, it suffices to show $f''_{\alpha, p}(x) \geq 0$ for
   $0<x \leq 1/2$. Assuming $0<x \leq 1/2$, we recast $f''_{\alpha,
   p}(x)$ as
\begin{align*}
   f''_{\alpha, p}(x)  & = \alpha p (1-x)^{-\alpha p
   -2}(1-(1-x)^{\alpha})^{p-2} \\
    & \cdot \Big ( g_{\alpha, p}(x)(\alpha p+1-(1+\alpha)x^{\alpha})+(\alpha
   p+1-(1+\alpha)(1-x)^{\alpha}) \Big ),
\end{align*}
   where
\begin{equation*}
  g_{\alpha, p}(x)=\Big ( \frac {1-x^{\alpha}}{x^{\alpha}} \cdot \frac {(1-x)^{\alpha}}{1-(1-x)^{\alpha}}  \Big
  )^p \cdot \Big ( \frac {1-x}{1-x^{\alpha}} \cdot \frac {1-(1-x)^{\alpha}}{1-(1-x)}  \Big
  )^2.
\end{equation*}
   It is easy to show that both factors of $g_{\alpha, p}(x)$ are
   $\geq 1$ and that $\alpha p+1-(1+\alpha)x^{\alpha} \geq 0$ when
   $0<x\leq 1/2$. We bound $g_{\alpha, p}(x)$ by
\begin{equation*}
  g_{\alpha, p}(x) \geq \frac {1-x}{1-x^{\alpha}} \cdot \frac {1-(1-x)^{\alpha}}{1-(1-x)}.
\end{equation*}
    It then follows that $f''_{\alpha, p}(x) \geq 0$ as long as
\begin{align*}
    \frac {1-x}{1-x^{\alpha}} \cdot \frac {1-(1-x)^{\alpha}}{1-(1-x)}(\alpha p+1-(1+\alpha)x^{\alpha})+(\alpha
   p+1-(1+\alpha)(1-x)^{\alpha}) \geq 0.
\end{align*}
   It suffices to establish the above inequality for $p=0$ and in
   this case we recast it as $h_{\alpha}(x)+h_{\alpha}(1-x) \geq
   0$ where
\begin{align*}
   h_{\alpha}(x)= \frac {1-x}{1-x^{\alpha}}(1-(1+\alpha)x^{\alpha})=(1+\alpha)(1-x)-\frac {\alpha(1-x)}{1-x^{\alpha}}.
\end{align*}
   It is easy to show that $h_{\alpha}(x)$ is concave on $(0,1)$
   and it follows from the theory of majorization (see, for example, Section 6 of \cite{B&J}) that for
   any $0<x<1$,  we have
\begin{equation*}
   h_{\alpha}(x)+h_{\alpha}(1-x) \geq \lim_{x \rightarrow
   0^+}(h_{\alpha}(x)+h_{\alpha}(1-x))=0.
\end{equation*}
   This now completes the proof of Theorem \ref{thm1.2} when
   $0<p<1, \alpha \geq 1, t=1$. Before we move to the proof of other
   cases, we point out here an alternative proof of $h_{\alpha}(x)+h_{\alpha}(1-x) \geq
   0$ is that one can show easily that $h_{\alpha}(x)$ is an increasing
   function of $\alpha \geq 1$ for fixed $x$ so that $h_{\alpha}(x)+h_{\alpha}(1-x) \geq \lim_{\alpha \rightarrow
   1^+}(h_{\alpha}(x)+h_{\alpha}(1-x))=0$. This will be our
   approach for the case $0<p \leq 1, 0<\alpha<1$ in what follows.

   Now the general case $0<p<1, \alpha \geq 1, 0 \leq t \leq 1$, we note that the left-hand side expression of \eqref{8} is termwise no larger
than the corresponding term when $t=1$. Therefore, inequality
\eqref{8} follows from the case $t=1$. To show the constant is
best possible, we use
    Theorem \ref{thm1} again to see that the best constant is given by $\max_{1 \leq r \leq
   N} s(t)_r$, where
\begin{equation*}
   s(t)_r=
   r^{-1}\sum^{r}_{k=1}\Big(\frac {(r+t)^{\alpha}-(k+t-1)^{\alpha}}{k^{\alpha}} \Big
  )^p \geq r^{-1}\sum^{r}_{k=1}\Big(\frac {r^{\alpha}-k^{\alpha}}{k^{\alpha}} \Big
  )^p.
\end{equation*}
   It follows that $\lim_{N \rightarrow \infty}s(t)_N \geq \frac {1}{\alpha}B(\frac {1}{\alpha}-p,p+1)$,
   this combining with our discussions above completes the proof
   of Theorem \ref{thm1.2} when $0<p<1, \alpha \geq 1$.

    For the case $0<p \leq 1, 0 < \alpha <1$, we can use the same approach as above except this time we bound $g_{\alpha, p}(x)$ by
\begin{equation*}
  g_{\alpha, p}(x) \geq \Big (\frac {1-x}{1-x^{\alpha}} \cdot \frac {1-(1-x)^{\alpha}}{1-(1-x)} \Big )^2.
\end{equation*}
    We define for $0<x<1$,
\begin{align*}
%%\label{4.0}
   u_{\alpha,p}(x)= \Big (\frac {1-x}{1-x^{\alpha}} \Big )^2(\alpha p+1-(1+\alpha)x^{\alpha}).
\end{align*}
    It remains to show $u_{\alpha,p}(x)+u_{\alpha,p}(1-x) \geq 0$ (note that this also implies that that $u_{\alpha,p}(1/2) \geq
   0$). On considering the limit as $x \rightarrow 0^+$, we see that it is
necessary to have $(1-\alpha)/(1+\alpha^2) \leq p \leq 1$. We now
assume this condition for $p$ and note that it suffices to
establish $u_{\alpha,p}(x)+u_{\alpha,p}(1-x) \geq 0$ for
$p=(1-\alpha)/(1+\alpha^2)$. We write $u_{\alpha,
    (1-\alpha)/(1+\alpha^2)}(x)=(1+\alpha)/(1+\alpha^2)v(\alpha,x)$
    with
\begin{equation*}
   v(\alpha, x)=(1-x)^2\frac {1-(1+\alpha^2)x^{\alpha}}{(1-x^{\alpha})^2}.
\end{equation*}
    It remains thus to show $v(\alpha, x)+v(\alpha, 1-x) \geq 0$. Calculation shows
\begin{equation*}
   \frac {\partial v}{\partial \alpha} =\frac
   {x^{\alpha}(1-x)^2}{\alpha(1-x^{\alpha})^3}w_{\alpha}(x^{\alpha}),
\end{equation*}
   with $w_{\alpha}(t)=-2\alpha^2+(1-\alpha^2)\ln t
   -(1+\alpha^2)t\ln t+2\alpha^2 t, 0<t \leq 1$. As
   $w_{\alpha}(1)=w'_{\alpha}(1)=0$ and $w''_{\alpha}(t)<0$, one
   sees easily that $w_{\alpha}(t) \leq 0$ for $0 < t \leq 1$ and it follows
   that when $0<x<1$,
\begin{equation*}
   v(\alpha, x)+v(\alpha, 1-x) \geq \lim_{\alpha \rightarrow 1^-}(v(\alpha, x)+v(\alpha,
   1-x))=0.
\end{equation*}
   This completes the proof of Theorem \ref{thm1.2} when $0<p \leq 1, 0< \alpha <1$.

   Lastly, when $p \geq 1$, the assertion of the theorem follows
   as long as we can show the sequence $(s_r)$ is increasing,
   where $(s_r)$ is defined as above. In this case, it's easy to
   see that the function $x \mapsto (1-x^{\alpha})^px^{-\alpha p}$
   is convex on $(0, 1)$ when $\alpha>0, p \geq 1$ so that our discussions
   above can be applied here and this completes the proof.

%%-------------------------------------------------------------------------------------
\section{Some consequences of Theorem \ref{thm1.2}}
\label{sec 4} \setcounter{equation}{0}
%%-------------------------------------------------------------------------------------
   In this section we deduce some consequences from Theorem \ref{thm1.2}. We
note that $(k+1)^{\alpha}-k^{\alpha} \geq \alpha
   k^{\alpha-1}$ when $\alpha \geq 1$ and it is also easy to show by induction that
\begin{equation*}
   \alpha \sum^r_{n=k}n^{\alpha-1} \geq r^{\alpha}-k^{\alpha}.
\end{equation*}
   A similar argument to the proof of
   Theorem \ref{thm1.2} then allows us to establish
\begin{cor}
\label{cor2}
  Let $0<p<1, \alpha \geq 1, \alpha p<1, x_1 \geq x_2 \geq \ldots \geq 0$. We
  have
\begin{equation}
\label{4.1}
   \sum^{\infty}_{n=1}\Big(\frac 1{n^{\alpha}} \sum^{\infty}_{k=n}\alpha k^{\alpha-1}x_k \Big
  )^p \leq \frac {1}{\alpha}B(\frac {1}{\alpha}-p,p+1)
  \sum^{\infty}_{n=1}x^p_n.
\end{equation}
  The constant is best possible.
\end{cor}

  We recall here the function $L_r(a,b)$ for $a>0, b>0, a \neq b$ and
$r \neq 0, 1$ (the only case we shall concern here) is defined
   as $L^{r-1}_r(a,b)=(a^r-b^r)/(r(a-b))$.
   We also write $L_{\infty}(a,b)$ as $\lim_{r \rightarrow \infty}L_{r}(a,b)$ and
   note that $L_{\infty}(a,b)=\max (a,b)$. Using this notation, the matrix $(a_{j,k})$ associated to inequality \eqref{4.1}
   is thus given by
   $a_{j,k}=L^{\alpha-1}_{\infty}(k,k-1)/\sum^j_{i=1}L^{\alpha-1}_{\alpha}(i,
   i-1)$ when $k \geq j$ and $a_{j,k}=0$ otherwise.

   It is known \cite[Lemma 2.1]{alz1.5} that
   the function $r \mapsto L_r(a,b)$ is strictly increasing on ${\mathbb
   R}$, this combining with Corollary \ref{cor2} allows us to establish the first assertion of the following
\begin{cor}
\label{cor3}
  Let $0<p<1, \beta \geq \alpha > 1, \alpha p<1, x_1 \geq x_2 \geq \ldots \geq 0$. We
  have
\begin{equation*}
   \sum^{\infty}_{n=1}\Big ( \frac
1{\sum^n_{i=1}L^{\alpha-1}_{\beta}(i,
i-1)}\sum^{\infty}_{k=n}L^{\alpha-1}_{\beta}(k, k-1)a_k \Big )^p
\leq \frac {1}{\alpha}B(\frac {1}{\alpha}-p,p+1)
  \sum^{\infty}_{n=1}x^p_n.
\end{equation*}
  The constant is best possible when $\alpha \geq 2$.
\end{cor}

   To show the constant is best possible when $\alpha \geq 2$, we
   first show that for $n \geq 1, \beta \geq \alpha \geq 2$,
\begin{equation}
\label{4.2}
   \frac
  {\sum^{n+1}_{i=1}L^{\alpha-1}_{\beta}(i, i-1)}{\sum^{n}_{i=1}L^{\alpha-1}_{\beta}(i, i-1)}
  \geq \frac {(n+2)^{\alpha}}{(n+1)^{\alpha}} .
\end{equation}
   By \cite[Lemma 3.1]{G7}, it suffices to show for $n \geq 1$,
\begin{equation*}
   \frac
  {L^{\alpha-1}_{\beta}(n+1, n)}{L^{\alpha-1}_{\beta}(n, n-1)}
  \geq \frac {(n+2)^{\alpha}-(n+1)^{\alpha}}{(n+1)^{\alpha}-n^{\alpha}} .
\end{equation*}
   The above inequality follows from the following inequalities:
\begin{equation}
\label{4.3}
   \frac
  {L^{\alpha-1}_{\beta}(n+1, n)}{L^{\alpha-1}_{\beta}(n, n-1)}
  \geq \frac
  {(n+1)^{\alpha-1}}{n^{\alpha-1}} \geq \frac {(n+2)^{\alpha}-(n+1)^{\alpha}}{(n+1)^{\alpha}-n^{\alpha}} .
\end{equation}
   As $\beta \geq 2$, we have by convexity,
\begin{equation*}
    \frac {1}{n}+\frac {n-1}{n}(\frac {n-1}{n})^{\beta-1} \geq
    (\frac {1}{n}+\frac {n-1}{n} \cdot \frac {n-1}{n})^{\beta-1}
    \geq
    (\frac {n}{n+1})^{\beta-1}.
\end{equation*}
   One checks easily that this implies the first inequality in
   \eqref{4.3} and the second inequality of \eqref{4.3} can be
   shown similarly.

   Now, to see the constant is best possible, we note that Theorem \ref{thm1} implies that the
    constant is no smaller than
\begin{eqnarray*}
    &&\lim_{N \rightarrow \infty}\frac {1}{N}\sum^{N}_{j=1}\Big(\frac
    {\sum^N_{k=j}L^{\alpha-1}_{\beta}(k, k-1)}{\sum^j_{i=1}L^{\alpha-1}_{\beta}(i, i-1)}\Big
    )^p\\
    &\geq&  \lim_{N \rightarrow \infty}\frac {1}{N}\sum^{N}_{j=1}\Big(\frac
    {\sum^N_{k=j+1}L^{\alpha-1}_{\beta}(k, k-1)}{\sum^j_{i=1}L^{\alpha-1}_{\beta}(i, i-1)}\Big
    )^p \geq \lim_{N \rightarrow \infty}\frac {1}{N}\sum^{N}_{j=1}\Big( \frac {(N+1)^{\alpha}}{(j+1)^{\alpha}}-1 \Big
    )^p,
\end{eqnarray*}
   where the last inequality follows from \eqref{4.2}
   and the last limit also
   gives the constant in Corollary \ref{cor3}.

    Note the particular case $\beta=\infty$ of Corollary \ref{cor3} gives
\begin{equation*}
%%\label{4.4}
   \sum^{\infty}_{n=1}\Big ( \frac
1{\sum^n_{i=1}i^{\alpha-1}}\sum^{\infty}_{k=n}k^{\alpha-1}a_k \Big
)^p \leq \frac {1}{\alpha}B(\frac {1}{\alpha}-p,p+1)
  \sum^{\infty}_{n=1}x^p_n.
\end{equation*}

%%-------------------------------------------------------------------------------------
\section{Applications of Theorem \ref{thm1} to weighted mean matrices}
\label{sec 5} \setcounter{equation}{0}
%%-------------------------------------------------------------------------------------
   In this section we give more applications of Theorem
   \ref{thm1}. We remark first that the problem of finding lower bounds of non-negative weighted mean
matrices acting on non-increasing non-negative sequences in $l^p$
when $p \geq 1$ has been studied in \cite{B2} and \cite{R&S}. Here
we recall that a weighted mean matrix $(a_{j,k})$ is given by
$a_{j,k}=\lambda_k/\Lambda_j$ for $1 \leq k \leq j$ and
$a_{j,k}=0$ otherwise, where $\Lambda_n=\sum^n_{i=1}\lambda_i,
\lambda_1>0$. We also recall that a N\"orlund matrix $(a_{j,k})$
is given by $a_{j,k}=\lambda_{j-k+1}/\Lambda_j$ for $1 \leq k \leq
j$ and $a_{j,k}=0$ otherwise. In what follows, we shall say a
weighted mean (or a N\"orlund) matrix  $A$ is generated by
$(\lambda_n)$ if its entries are given as above. In the weighted
mean matrix case, it is shown in \cite[Theorem 4]{B2} that when
$\lambda_i=i^{\alpha}, \alpha \geq 1$ or $-1<\alpha \leq 0,
(1+\alpha)p>0$, the corresponding minimum in \eqref{3} is reached
at $r=1$. The case $\alpha \geq 1$ is also shown in
\cite[Corollary 9]{R&S}. We now give an alternative proof of the
case $-1< \alpha \leq 0$ based on the idea used in the proof of
Theorem 4 in \cite{B}. We also give a companion result concerning
the upper bound when $0<\alpha \leq 1$ and $0<p \leq 1$. We have
\begin{cor}
\label{cor4.3}
   Let ${\bf x}$ be a non-negative non-increasing sequence, $p > 1$,
  $-1< \alpha \leq 0, (\alpha+1)p>1$, then
\begin{equation}
\label{4.8}
  \sum^{\infty}_{j=1}\Big(\sum^j_{k=1}\frac {k^{\alpha}}{\sum^j_{i=1}i^{\alpha}}x_k \Big )^p
  \geq \sum^{\infty}_{j=1}\Big( \frac {1}{\sum^j_{i=1}i^{\alpha}} \Big )^p ||{\bf
  x}||^p_p.
\end{equation}
   The above inequality reverses
   when $0<p \leq 1$, $0<\alpha \leq 1, (1+\alpha)p>1$. The constant is best possible in either case.
\end{cor}
\begin{proof}
   Note that the condition $(\alpha+1)p>1$ ensures that the constant in \eqref{4.8} is finite.
   We consider the case $p > 1$ first. For any weighted mean matrix $A$ generated by $(\lambda_n)$
   with $\lambda_1>0$, Theorem \ref{thm1}
implies that for any non-increasing sequence ${\bf x}$, $||A{\bf
x}||_p \geq \lambda ||{\bf x}||_p$ with
\begin{eqnarray*}
   \lambda^p &=& \inf_{r}
   r^{-1}\sum^{\infty}_{j=1}(\sum^{\min (r,j)}_{k=1}\frac {\lambda_{k}}{\Lambda_j})^p=1+\inf_{r}
   r^{-1}\sum^{\infty}_{j=r+1}\Big(\frac {\Lambda_{r}}{\Lambda_j} \Big )^p \\
    &=& 1+\inf_{r}
   \sum^{\infty}_{k=1}r^{-1} \sum^{r}_{i=1}\Big(\frac
   {\Lambda_{r}}{\Lambda_{kr+i}}\Big)^p.
\end{eqnarray*}
   To show the infimum is achieved at $r=1$, it suffices to
   show $\Lambda_r / \Lambda_{(k+1)r} \geq \Lambda_1 /
   \Lambda_{k+1}$ for any $k \geq 1$. When
   $\lambda_n=n^{\alpha}$ with $-1<\alpha \leq 0$, this is easily shown
   by induction and this completes the proof for the first assertion of the corollary.

   Now consider the case $0<p \leq 1$. Let $\Lambda_i=\sum^i_{j=1}j^{\alpha}$. Theorem \ref{thm1} and our discussions above
    imply that the best constant for the reversed inequality of \eqref{4.8} is given by
\begin{equation*}
   1+\sup_{r}
   r^{-1}\sum^{\infty}_{j=r+1}\Big(\frac {\Lambda_{r}}{\Lambda_j} \Big )^p=1+\sup_ra_r.
\end{equation*}
   The assertion of the corollary follows if we can show the sequence $(a_r)$
   is decreasing and by Lemma 7 of \cite{B2} (see the remark after that) with $x_n=\Lambda^{-p}_n$ there, it suffices to
   show $1+n(\Lambda_{n+1}/\Lambda_n)^p \leq (n+1)(\Lambda_{n+2}/\Lambda_{n+1})^p$ for $n \geq 1$
   and one can see easily that it suffices to establish this for $p=1$
   but in this case, this is given by Lemma 8 of \cite{B2} and this completes the proof.
\end{proof}

   Our next result concerns with the bounds for the
   weighted mean matrix generated by $\lambda_i=i^{\alpha}-(i-1)^{\alpha}, \alpha
   p>1$:
\begin{cor}
\label{cor4.04}
   Let ${\bf x}$ be a non-negative non-increasing sequence, $p \geq 1$,
  $\alpha >1/p$, then
\begin{equation}
\label{4.08}
  \sum^{\infty}_{j=1}\Big(\sum^j_{k=1}\frac {k^{\alpha}-(k-1)^{\alpha}}{j^{\alpha}}x_k \Big )^p
  \geq \zeta(\alpha p) ||{\bf
  x}||^p_p,
\end{equation}
   where $\zeta(x)$ denotes the Riemann zeta function and the constant is best possible. The above inequality reverses
   when $0<p \leq 1, \alpha p>1$ with the best constant $\alpha
   p/(\alpha p-1)$.
\end{cor}
\begin{proof}
   The proof for the $p \geq 1$ case can be easily obtained by applying similar ideas to that used in the proof of
Theorem 4 in \cite{B} so we shall leave it to the reader. When
$0<p \leq 1$, we note by Theorem \ref{thm1} and the proof of
Corollary \ref{cor4.3}, the best constant for the reversed
inequality of \eqref{4.08} is given by
\begin{eqnarray*}
    1+\sup_{r}
   \sum^{\infty}_{k=1}r^{-1} \sum^{r}_{i=1}\Big(\frac
   {\Lambda_{r}}{\Lambda_{kr+i}}\Big)^p=1+\sup_{r}
   \sum^{\infty}_{k=1}\sum^{r}_{i=1}\frac {(k+i/r)^{-\alpha p}}{r}.
\end{eqnarray*}
  It follows from Theorem 3A of \cite{B&J} that the term inside the first sum of the last expression above is increasing with
  $r$, and it is easy to see that as $r \rightarrow +\infty$, it
  approaches the value $(k^{1-\alpha p}-(k+1)^{1-\alpha
  p})/(\alpha p-1)$ and this completes the proof.
\end{proof}

    It is an open problem to determine the lower bounds of the weighted mean matrices generated by $\lambda_n=n^{\alpha}, 0<\alpha<1$
    acting on non-increasing non-negative sequences in $l^p$ when $p \geq
    1$. In connection to this, Bennett \cite[p. 65]{B2} asked to
    determine the monotonicity of the following sequence for $p>1,
    (1+\alpha)p>1$, when $\Lambda_n=\sum^n_{i=1}i^{\alpha}$:
\begin{equation}
\label{5.2}
    \frac {\Lambda^p_n}{n}\sum_{k>n}\Lambda^{-p}_k.
\end{equation}
    The following condition is sufficient
    for the above sequence to be increasing, given by \cite[Theorem 3]{B2} (see also \cite[Theorem
    8]{R&S}):
\begin{equation}
\label{4.7}
    1+n\Big (\frac {\Lambda_{n+1}}{\Lambda_n}\Big )^p-(n+1)\Big(\frac {\Lambda_{n+2}}{\Lambda_{n+1}}\Big
    )^p \geq 0.
\end{equation}
    Suppose the above condition is satisfied, then we deduce from
    it that
\begin{equation*}
   n\Big (\frac {\Lambda_{n+1}}{\Lambda_n}\Big )^p \geq (n+1)\Big(\frac {\Lambda_{n+2}}{\Lambda_{n+1}}\Big
    )^p-1 \geq (n+2)\Big(\frac {\Lambda_{n+3}}{\Lambda_{n+2}}\Big
    )^p-2 \geq \ldots \geq (n+k)\Big(\frac {\Lambda_{n+k+1}}{\Lambda_{n+k}}\Big
    )^p-k,
\end{equation*}
    for any $n, k \geq 1$. When
    $\Lambda_n=\sum^n_{i=1}i^{\alpha}$, we note by the Euler-Maclaurin
    formula, one easily finds that for $\alpha >0$,
\begin{equation}
\label{5.3}
   \sum^n_{i=1}i^{\alpha}=\frac {n^{\alpha+1}}{1+\alpha}+\frac
   {n^{\alpha}}{2}+O(1+n^{\alpha-1}).
\end{equation}
    We then deduce from this that when
    $\Lambda_n=\sum^n_{i=1}i^{\alpha}, \alpha>0$,
\begin{equation*}
    \lim_{k \rightarrow +\infty}(n+k)\Big(\frac {\Lambda_{n+k+1}}{\Lambda_{n+k}}\Big
    )^p-k=n+(1+\alpha)p.
\end{equation*}
    It follows from this that we have
\begin{equation*}
    \frac {\Lambda_{n+1}}{\Lambda_n} \geq (1+\frac
    {(1+\alpha)p}{n})^{1/p}.
\end{equation*}
    Using Taylor expansion and \eqref{5.3} again, we find that in
    order for the above inequality to hold, it is necessary to
    have $p \geq 2/(1+\alpha)$. It is therefore interesting to
    ask whether the above inequality holds or not for
    $p=2/(1+\alpha)$ and this in fact is known, as we have the
    following
\begin{lemma}
\label{lem4}
   For $0 \leq \alpha \leq 1$ or $\alpha \geq 3$, we have
   for $k \geq n \geq 1$,
\begin{equation}
\label{4.22}
   \frac {\sum^n_{i=1}i^{\alpha}}{\sum^k_{i=1}i^{\alpha}} \leq
   \Big ( \frac {n(n+1)}{k(k+1)} \Big )^{\frac {\alpha+1}{2}}.
\end{equation}
   The above inequality reverses when $1 \leq \alpha \leq 3$. In
   particular, we have for $0 \leq \alpha \leq 1$ or $\alpha \geq 3$,
\begin{equation}
\label{4.25}
   \sum^n_{i=1}i^{\alpha} \leq  \frac {(n(n+1))^{\frac
   {\alpha+1}{2}}}{\alpha+1}.
\end{equation}
   The above inequality reverses when $1 \leq \alpha \leq 3$.
\end{lemma}
\begin{proof}
   This lemma is a restatement of Corollary 3.1 of \cite{G7} (note that in the statement of \cite[Corollary 3.1]{G7}, one needs to
   interchange the place of the words ``increasing" and ``decreasing"). In what follows, we shall give a simper proof. We first note that inequality \eqref{4.25} follows from the corresponding cases of \eqref{4.22} on
   letting $k \rightarrow +\infty$ in \eqref{4.22} so that it suffices to establish \eqref{4.22}.
   We shall only prove the case for $\alpha \geq 3$, the proof for the other cases are similar.
   We may assume $k =n+1$ here and by Lemma 3.1 of \cite{G7}, it
   suffices to establish \eqref{4.22} for $n=1$ as well as the
   following inequality for all $n \geq 1$:
\begin{equation}
\label{4.23}
   \frac {(n+1)^{\alpha}}{(n+2)^{\alpha}} \leq
   \frac {\Big((n+2)(n+1)\Big )^{(1+\alpha)/2}-\Big(n(n+1)\Big )^{(1+\alpha)/2}}
   {\Big((n+3)(n+2)\Big)^{(1+\alpha)/2}-\Big
   ((n+1)(n+2)\Big)^{(1+\alpha)/2}}.
\end{equation}
   The above inequality is easily seen to be equivalent to $f(n+2)
   \leq f(n+1)$ where
\begin{equation*}
    f(x)=\frac
    {(x+1)^{(1+\alpha)/2}-(x-1)^{(1+\alpha)/2}}{x^{(\alpha-1)/2}}=\frac {1+\alpha}{2}\int^1_0\Big
    ((1+t/x)^{(\alpha-1)/2}+(1-t/x)^{(\alpha-1)/2}\Big )dt.
\end{equation*}
   One shows easily the last expression above is a
   decreasing function of $x \geq 1$ when $\alpha \geq 3$ so that \eqref{4.23} holds. Moreover, the case $n=1, k=2$ of \eqref{4.22} is just
   $f(2) \leq f(1)$ and this completes the proof.
\end{proof}

   The above lemma implies that (in combining the arguments given in \cite[Theorem 3]{B2} or \cite[Theorem
    8]{R&S}) the sequence defined in \eqref{5.2} for $\Lambda_n=\sum^n_{i=1}i^{\alpha}$ is increasing for $n$ large enough
    when $0 \leq \alpha \leq 1$ or $\alpha \geq 3$, as long as $p \geq 2/(1+\alpha)$ and it is decreasing for $n$ large enough
    when $1 \leq \alpha \leq 3$, as long as $1/(1+\alpha)<p \leq 2/(1+\alpha)$. It's also shown in \cite{B2} that the sequence is increasing
    for $\alpha \geq 1, p \geq 1$ and decreasing for $0<\alpha
    \leq 1, 1/(1+\alpha) < p \leq 1$. In what follows, we shall give an extension of this result. But we first need a few lemmas:
\begin{lemma}
\label{lem5}
   For $1 \leq \alpha \leq 3$, we have
\begin{equation}
\label{5.15}
    \sum^n_{i=1}i^{\alpha} \geq \frac {1}{1+\alpha}\frac
    {4n^2(n+1)^{\alpha}}{4n+1+\alpha}.
\end{equation}
   The above inequality reverses when $\alpha \geq 3$.
\end{lemma}
\begin{proof}
   We only give the proof for the case $1 \leq \alpha \leq 3$ and the proof for the other case is similar.
   It follows from \eqref{4.22} with $k=n+1$ that we have for $1 \leq \alpha
   \leq 3$,
\begin{equation}
\label{5.12}
   \frac {\sum^n_{i=1}i^{\alpha}}{\sum^{n+1}_{i=1}i^{\alpha}} \geq
   \Big ( \frac {n}{n+2} \Big )^{\frac {\alpha+1}{2}}.
\end{equation}

    We deduce from this that for $1 \leq \alpha \leq 3$,
\begin{equation}
\label{5.13}
    \sum^n_{i=1}i^{\alpha} \geq \frac
    {n^{(1+\alpha)/2}(n+1)^{\alpha}}{(n+2)^{(1+\alpha)/2}-n^{(1+\alpha)/2}}.
\end{equation}
     It suffices to show the right-hand side expression above is
     no less than the right-hand side expression of \eqref{5.15}.
     One easily sees that this follows from the following
     inequality for $1 \leq \alpha \leq 3, 0 \leq x \leq 1$:
\begin{equation*}
     1+(1+\alpha)x+\frac {(1+\alpha)^2x^2}{4}-(1+2x)^{(1+\alpha)/2}
     \geq 0.
\end{equation*}
    The above inequality can be shown easily and this completes
    the proof.
\end{proof}

\begin{lemma}
\label{lem6}
    Let $0 \leq x \leq 1$, the following inequality holds when $1
    \leq \alpha \leq 3$:
\begin{equation}
\label{5.14}
   \Big((1+x)^{2-\alpha}(1+2x)^{(\alpha-1)/2}-1\Big)\Big((1+2x)^{(1+\alpha)/2}-1 \Big )-(1+\alpha)x^2 \geq 0.
\end{equation}
   The above inequality reverses when $\alpha \geq 3$.
\end{lemma}
\begin{proof}
    We regard the left-hand side expression of
   \eqref{5.14} as a function of $\alpha$ and note
   that
its second derivative with respect to $\alpha$ equals
   $(1+2x)^{(\alpha-1)/2}h(\alpha;x)$, where
\begin{eqnarray*}
   h(\alpha;x) &=& \ln^2\Big ( \frac {1+2x}{1+x}
   \Big)(1+x)^{2-\alpha}(1+2x)^{(\alpha+1)/2}-\Big ( \frac {\ln (1+2x)}{2}
   \Big)^2(1+2x) \\
   &&-\ln^2\Big ( \frac {(1+2x)^{1/2}}{1+x}
   \Big)(1+x)^{2-\alpha}.
\end{eqnarray*}
    We again regard $h(\alpha;x)$ as a function of $\alpha$ and
    note that
\begin{equation*}
    h'(\alpha;x)=(1+x)^{2-\alpha}\ln \Big ( \frac
    {(1+2x)^{1/2}}{1+x}\Big )\Big (\ln^2\Big ( \frac {1+2x}{1+x}
   \Big)(1+2x)^{(\alpha+1)/2}+ \ln \Big ( \frac
    {(1+2x)^{1/2}}{1+x}\Big )\ln(1+x) \Big ).
\end{equation*}
   We want to show the last factor of the right-hand side
   expression above is non-negative when $\alpha \geq 1$ and it suffices to show this for
   $\alpha=1$ and in this case, this expression becomes
\begin{eqnarray*}
   &&\ln^2\Big ( \frac {1+2x}{1+x}
   \Big)(1+2x)+ \ln \Big ( \frac
    {(1+2x)^{1/2}}{1+x}\Big )\ln(1+x) \\
   & = & (1+2x)\ln^2(1+2x)-(2(1+2x)-1/2)\ln (1+2x)\ln
   (1+x)+2x\ln^2(1+x) \\
   & \geq & (1+2x)\ln^2(1+2x)-(2(1+2x)-1/2)\ln (1+2x)\ln
   (1+x)+x\ln (1+x) \ln (1+2x) \\
   &=& (1+2x)\ln (1+2x)\Big(\ln (1+2x)-3\ln (1+x)/2 \Big ) \geq 0.
\end{eqnarray*}
   It follows that $h'(\alpha;x) \leq 0$. As the left-hand side
   expression of \eqref{5.14} takes value $0$ when $\alpha=1$ and $3$, the
   assertion of the lemma follows if we can show the derivative
   with respect to $\alpha$ of the left-hand side expression of
   \eqref{5.14} is $\geq 0$ ($>0$ for $x \neq 0$) at $\alpha=1$. Calculation shows this
   is
\begin{eqnarray*}
  && \ln\Big ( \frac {1+2x}{1+x}
   \Big)(1+x)(1+2x)-\frac {\ln (1+2x)}{2}(1+2x)-\ln\Big ( \frac {(1+2x)^{1/2}}{1+x}
   \Big)(1+x)-x^2 \\
  &=& x\Big ((3/2+2x)\ln(1+2x)-2(1+x)\ln (1+x)-x \Big ).
\end{eqnarray*}
   It is easy to show the second factor in the last expression
   above is $\geq 0$ for $0 \leq x \leq 1$ ( $>0$ for $x \neq 0$) and this completes the
   proof.
\end{proof}

    Now we are ready to prove the following
\begin{theorem}
\label{thm3}
   For $1 < \alpha \leq 3$ and $1/(1+\alpha)<p \leq 1/2$, the sequence defined in \eqref{5.2} for $\Lambda_n=\sum^n_{i=1}i^{\alpha}$ is
   decreasing. For $\alpha \geq 3$ and $p \geq 1/2$, the sequence defined in \eqref{5.2} for $\Lambda_n=\sum^n_{i=1}i^{\alpha}$ is
   increasing.
\end{theorem}
\begin{proof}
   We only prove the case for $1 \leq \alpha \leq 3$ here and the proof for the case $\alpha \geq 3$ is
   similar. We only point out that in the $\alpha \geq 3$ case,
   one needs to use the fact (which is easy to show) that the right-hand side expression of
   \eqref{5.12} is no greater than $n(n+1)^{\alpha}/(1+\alpha)$ (and hence $\leq
   n(n+1)^{\alpha}/\sqrt{1+\alpha}$). Now we return to the proof
   of our assertion for $1 \leq \alpha \leq 3$ and by the remark after Lemma 7 of \cite{B2} (with
$x_n=\Lambda^{-p}_{n}$ there), it suffices to prove the reversed
inequality of
   \eqref{4.7} for $p=1/2$, which is equivalent to
\begin{equation}
\label{5.9}
  2n \Big(\frac {\Lambda_{n+1}}{\Lambda_n} \Big)^{1/2} \leq \frac
  {(n+2)^{\alpha}(n+1)^2}{\Lambda_{n+1}}-\frac
  {(n+1)^{\alpha}n^2}{\Lambda_{n}}+2n.
\end{equation}

   We now show for $1 \leq \alpha \leq 3$, we have
\begin{equation}
\label{5.10}
 \frac
  {(n+2)^{\alpha}(n+1)^2}{\Lambda_{n+1}}-\frac
  {(n+1)^{\alpha}n^2}{\Lambda_{n}} \geq 1+\alpha.
\end{equation}
   We recast this as
\begin{equation}
\label{5.11}
   (n+2)^{\alpha}(n+1)^2-(n+1)^{\alpha}n^2 \geq
(1+\alpha)(n+1)^{\alpha}+(1+\alpha)\Lambda_n+\frac
  {(n+1)^{2\alpha}n^2}{\Lambda_{n}}.
\end{equation}
   We now regard $\Lambda_n$ as a variable on the right-hand side
   expression above and it is easy to see this is a convex
   function with the unique critical point being
   $n(n+1)^{\alpha}/\sqrt{1+\alpha}$. Note that we have
   $\sum^{n}_{i=1}i^{\alpha} \leq n(n+1)^{\alpha}/(1+\alpha)$ for
   $\alpha \geq 1$ (this follows from \cite[Lemma 8]{B2}).
   It follows that it suffices to establish \eqref{5.11} with $\Lambda_n$
   replaced by the lower bound given in \eqref{5.13}. Equivalently, we can then multiply both sides of
   \eqref{5.10} by $\Lambda_n$ and in the resulting expression replace the values of $\Lambda_n/\Lambda_{n+1}$ and $\Lambda_n$ by the values given
   by the right-hand side expressions of \eqref{5.12} and \eqref{5.13} respectively. Then after some simplifications and on setting
   $x=1/n$, we see that inequality \eqref{5.10}
   is a consequence of inequality \eqref{5.14} for $0 \leq x \leq
   1$. Substituting \eqref{5.10} in \eqref{5.9} and squaring both sides, we
    find that it suffices to show \eqref{5.15} and Lemma
    \ref{lem5} now leads to the assertion of the theorem.
\end{proof}

   We now apply our results above to prove the
following
\begin{theorem}
\label{thm4}
   Let ${\bf x}$ be a non-negative non-increasing sequence, $0< p  \leq 1$,
  $\alpha \geq 3, (\alpha+1)p>2$, then
\begin{equation}
\label{4.24}
  \sum^{\infty}_{j=1}\Big(\sum^j_{k=1}\frac {k^{\alpha}}{\sum^j_{i=1}i^{\alpha}}x_k \Big )^p
  \leq \frac {(1+\alpha)p}{(1+\alpha)p-1} ||{\bf
  x}||^p_p.
\end{equation}
   The constant is best possible. The above inequality also holds
   when $1 < \alpha \leq 3, 1/(1+\alpha) < p \leq 1/2$ with the best possible
   constant $\sum^{\infty}_{j=1}\Big( \sum^j_{i=1}i^{\alpha} \Big
   )^{-p}$.
\end{theorem}
\begin{proof}
   The second assertion of the theorem is a direct consequence of Theorems \ref{thm1} and \ref{thm3}.
   To prove the first assertion of the theorem, we let $\Lambda_{n, \alpha}=\sum^n_{i=1}i^{\alpha}$ and Theorem
\ref{thm1} implies that the best constant in
   \eqref{4.24} is given by
\begin{equation*}
   1+\sup_{r}
   r^{-1}\sum^{\infty}_{j=r+1}\Big(\frac {\Lambda_{r, \alpha}}{\Lambda_{j, \alpha}} \Big
   )^p \leq 1+\sup_{r}
   r^{-1}\sum^{\infty}_{j=r+1}\Big(\frac {\Lambda_{r,1}}{\Lambda_{j,1}} \Big
   )^{\frac {(1+\alpha)p}{2}}=1+\sup_rb_r,
\end{equation*}
   by Lemma \ref{lem4}. We want to show $(b_r)$ is increasing and by Lemma 7 of \cite{B2} with $x_n=\Lambda^{-(1+\alpha)p/2}_{n,1}$ there, it suffices to
   show $1+n(\Lambda_{n+1,1}/\Lambda_{n,1})^{(1+\alpha)p/2} \geq (n+1)(\Lambda_{n+2,1}/\Lambda_{n+1,1})^{(1+\alpha)p/2}$ for $n \geq 1$
   and one sees easily that it suffices to establish this for
   $(1+\alpha)p=2$, in which case the inequality becomes an
   identity. It follows that $\sup_rb_r=\lim_{r \rightarrow
   +\infty}b_r$ and we note that
\begin{equation*}
   b_r=\sum^{\infty}_{k=1}\frac {1}{r}\sum^r_{i=1}\Big(\frac {r(r+1)}{(kr+i)(kr+i+1)} \Big
   )^{\frac {(1+\alpha)p}{2}}=\sum^{\infty}_{k=1}\frac {1}{r}\sum^r_{i=1}\Big(
   (k+i/r)^{-(1+\alpha)p}+O(1/r)\Big
   ).
\end{equation*}
   It follows that as $r \rightarrow +\infty$, the inner sum of
   the last expression above approaches the value $(k^{1-(\alpha+1)
   p}-(k+1)^{1-(\alpha+1)
  p})/((1+\alpha)p-1)$ so that $\lim_{r \rightarrow
  +\infty}b_r=1/((1+\alpha)p-1)$. We then deduce that the best constant in
   \eqref{4.24} is $\leq (1+\alpha)p/((1+\alpha)p-1)$.
   On the other hand, the first inequality of \cite[(1.3)]{G7} implies that
   $\Lambda_{r, \alpha}/\Lambda_{j, \alpha} \geq (r/j)^{1+\alpha}$ when $j \geq r$
   so that Corollary \ref{cor4.04} implies that the best constant in
   \eqref{4.24} is $\geq (1+\alpha)p/((1+\alpha)p-1)$. This now
   completes the proof.
\end{proof}

%%----------------------------------------------------------------------------------------------------------
%%----------------------------------------------------------------------------------------------------------
   We now return to the question of determining the monotonicity of the sequence given in \eqref{4.7}
   for $\Lambda_n=\sum^n_{i=1}i^{\alpha}$ and note that the most
   interesting case here is $0<\alpha<1<p$ (see \cite[p. 65]{B2}),
   in view of the connection to the open problem of determining
   the lower bounds of the weighted mean matrices generated by $\lambda_n=n^{\alpha}, 0<\alpha<1$
    acting on non-increasing non-negative sequences in $l^p$ when $p \geq
    1$. In what follows, we shall give a partial solution to this
    and we point out here that we have not tried to optimize the
    choice of the auxiliary function appearing in the proof of
    Theorem \ref{thm5} and one may be able to obtain better lower bounds for
    $\alpha$ appearing in Theorem  \ref{thm5} as well as Corollary
    \ref{cor5.3}.

    We now prove a few lemmas:
\begin{lemma}
\label{lem7}
    Let $\Lambda_n=\sum^n_{i=1}i^{\alpha}$. For $0.14 \leq \alpha \leq
   1$, $n \geq 1$, we have
\begin{equation}
\label{5.18}
   \frac {n(n+1)^{2\alpha}}{\Lambda^2_n}-\frac
   {(n+1)(n+2)^{2\alpha}}{\Lambda^2_{n+1}}-\frac {0.94(1+\alpha)}{(n+1)^2}
   \geq 0.
\end{equation}
\end{lemma}
\begin{proof}
    We first prove inequality \eqref{5.18} holds when $n=1$ for all $0 \leq \alpha \leq 1$.
    In fact we shall prove the following stronger inequality:
\begin{equation*}
    2^{2\alpha}-\frac {2\cdot 3^{2\alpha}}{(1+2^{\alpha})^2}-\frac
    {1+\alpha}{2} \geq 0.
\end{equation*}
    Now using the bound $1+2^{\alpha} \geq 2^{1+\alpha/2}$, we see
    that the above inequality is a consequence of the following
    inequality:
\begin{equation*}
   2^{2\alpha+1}-(9/2)^{\alpha}-1-\alpha \geq 0.
\end{equation*}
   It is easy to show that the left-hand side expression above, as
   a function of $\alpha$, $0 \leq \alpha \leq 1$, is convex and increasing and as it takes the value $0$ at $\alpha=0$, this
   completes the proof for the case $n=1$ of \eqref{5.18}.

   Now we assume $n \geq 2$ and note that the reversed inequalities \eqref{5.12} and \eqref{5.13}
are still valid when $0 \leq \alpha \leq 1$ and it is easy to see,
using the reversed inequality of \eqref{5.13} for $0 \leq \alpha
\leq 1$, that the left-hand side expression of \eqref{5.18} is a
decreasing function of $\Lambda_n$ and hence it suffices to
establish \eqref{5.18} on multiplying both sides of
   \eqref{5.16} by $\Lambda^2_n$ and in the resulting expression replacing the values of $\Lambda_n/\Lambda_{n+1}$ and $\Lambda_n$ by the values given
   by the right-hand side expressions of \eqref{5.12} and \eqref{5.13}
   respectively. We are now led to show the following inequality:
\begin{equation*}
   n(n+1)^{2\alpha} \geq (n+1)(n+2)^{2\alpha}\Big ( \frac
   {n}{n+2} \Big )^{1+\alpha}+\frac
   {0.94(1+\alpha)}{(n+1)^2}n^{1+\alpha}(n+1)^{2\alpha}\Big ( (n+2)^{(\alpha+1)/2}-n^{(\alpha+1)/2} \Big
   )^{-2}.
\end{equation*}
    After some simplifications and on setting
   $x=1/n$, we can recast the above inequality as
\begin{equation}
\label{5.19}
   1 \geq (1+x)^{1-2\alpha}(1+2x)^{\alpha-1}+\frac
   {0.94(1+\alpha)x^3}{(1+x)^2}\Big ( (1+2x)^{(\alpha+1)/2}-1 \Big
   )^{-2}.
\end{equation}
   By Hadamard's inequality, which asserts for a continuous convex function $h(u)$ on $[a, b]$,
\begin{equation*}
    \frac {1}{b-a}\int^b_ah(u)du \geq h(\frac {a+b}{2}),
\end{equation*}
   we see that
\begin{equation*}
    (1+2x)^{(\alpha+1)/2}-1 = \frac
    {1+\alpha}{2}\int^{2x}_0(1+t)^{(\alpha-1)/2}dt \geq
    (1+\alpha)x(1+x)^{(\alpha-1)/2}.
\end{equation*}
    It suffices to prove \eqref{5.19} with $(1+2x)^{(\alpha+1)/2}-1
    $ replaced by this lower bound above which leads to the
    following inequality (with $x=1/n$) for $0 \leq x \leq 1/2$:
\begin{equation}
\label{5.21}
    1 \geq \frac {1+x}{1+2x}\Big ( \frac{1+2x}{(1+x)^2} \Big )^{\alpha}+\frac
   {0.94x(1+x)^{-1-\alpha}}{(1+\alpha)}.
\end{equation}
    Note that we have
\begin{eqnarray*}
   && \frac {1+x}{1+2x}\Big ( \frac{1+2x}{(1+x)^2} \Big )^{\alpha}+\frac
   {0.94x(1+x)^{-1-\alpha}}{(1+\alpha)} \\
   &=& \frac {1+x}{1+2x}\Big ( \frac{1+2x}{(1+x)^2} \Big )^{\alpha}+\frac {x}{1+2x}\Big (\Big (\frac
   {0.94(1+2x)}{(1+\alpha)(1+x)^{1+\alpha}} \Big )^{1/\alpha} \Big )^{\alpha} \\
   & \leq & \Big ( \frac {1+x}{1+2x} \cdot \frac{1+2x}{(1+x)^2}+ \frac {x}{1+2x} \cdot \Big (\frac
   {0.94(1+2x)}{(1+\alpha)(1+x)^{1+\alpha}} \Big )^{1/\alpha} \Big )^{\alpha}.
\end{eqnarray*}
    Hence it suffices to show the last expression above is $\leq
    1$, which is equivalent to showing for $\alpha \geq 0.14, 0 \leq x \leq 1/2$,
\begin{equation}
\label{5.22}
   (1+\alpha)(1+x)-0.94(1+2x)^{1-a} \geq 0.
\end{equation}
    To see this, observe that the left-hand side expression above
    is an increasing function of $\alpha$, hence it suffices to
    check the above inequality for $\alpha=0.14$, in which case we also observe that the left-hand side
    expression above is a convex function of $x$ and its derivative at $x=1/2$ is negative. It follows that
    one only needs to check the case when $x=1/2$ and one checks
    easily that \eqref{5.22} holds in this case. This now establishes inequality \eqref{5.21} and hence completes the
   proof.
\end{proof}

\begin{lemma}
\label{lem8}
   Let $\Lambda_n=\sum^n_{i=1}i^{\alpha}$. For $0.14 \leq \alpha \leq
   1$, $n \geq 1$, we have
\begin{equation}
\label{5.16}
   \frac {2n(n+1)^{\alpha}}{\Lambda_n}-\frac
   {2(n+1)(n+2)^{\alpha}}{\Lambda_{n+1}}+\frac {0.94(1+\alpha)}{(n+1)^2}
   \geq 0.
\end{equation}
\end{lemma}
\begin{proof}
    We first prove inequality \eqref{5.16} holds when $n=1$ for all $0 \leq \alpha \leq
    1$, in which case the inequality becomes
\begin{equation*}
    2^{\alpha+1}-\frac {4\cdot 3^{\alpha}}{1+2^{\alpha}}+\frac
    {0.94(1+\alpha)}{4} \geq 0.
\end{equation*}
    Now using the bound $1+2^{\alpha} \geq 2^{1+\alpha/2}$, we see
    that the above inequality is a consequence of the following
    inequality:
\begin{equation*}
   2^{\alpha+1}-2\cdot(3/\sqrt{2})^{\alpha}+\frac
    {0.94(1+\alpha)}{4} \geq 0.
\end{equation*}
   It is easy to show that the left-hand side expression above, as
   a function of $\alpha$, $0 \leq \alpha \leq 1$, is concave so that it suffices to check its values at $\alpha=0$ and $\alpha=1$, in both
   cases the above inequality can be verified easily and this
   completes the proof for the case $n=1$ of \eqref{5.16}.

   Now assume $n \geq 2$ and we recast inequality \eqref{5.16} as
\begin{equation*}
   \frac {2n(n+1)^{2\alpha}}{\Lambda_n}+\frac {0.94(1+\alpha)}{(n+1)^2}\Lambda_n+2n(n+1)^{\alpha}-2(n+1)(n+2)^{\alpha}+\frac
   {0.94(1+\alpha)}{(n+1)^2}(n+1)^{\alpha}
   \geq 0.
\end{equation*}

   We now regard $\Lambda_n$ as a variable on the left-hand side
   expression above and it is easy to see this is a convex
   function with the unique critical point being
   $\sqrt{(2n)}(n+1)^{\alpha+1}/\sqrt{0.94(1+\alpha)}$. Note that the
   reversed inequalities \eqref{5.12} and \eqref{5.13} are still valid when
   $0 \leq \alpha \leq 1$ and we want to show first that the upper
   bound given in the reversed inequality in \eqref{5.13} for
   $\Lambda_n$ is no greater than
   $\sqrt{(2n)}(n+1)^{\alpha+1}/\sqrt{0.94(1+\alpha)}$. In fact it
   suffices to show it is no greater than
$\sqrt{2}n(n+1)^{\alpha+1/2}/\sqrt{1+\alpha}$, which is
   equivalent to showing the following inequality
\begin{equation}
\label{5.17}
   n^{(\alpha-1)/2} \leq \frac
   {\sqrt{2}}{\sqrt{1+\alpha}}(n+1)^{1/2}\Big((n+2)^{(\alpha+1)/2}-n^{(\alpha+1)/2}\Big ).
\end{equation}
    Note that it follows from the mean value theorem, we have
    $(n+2)^{(\alpha+1)/2}-n^{(\alpha+1)/2}\geq
    (1+\alpha)(n+2)^{(\alpha-1)/2}$. Using this in \eqref{5.17},
    we see that it remains to show
\begin{equation*}
   (1+2/n)^{(1-\alpha)/2} \leq \sqrt{2(1+\alpha)}(n+1)^{1/2}.
\end{equation*}
   But we have $(1+2/n)^{(1-\alpha)/2} \leq (1+2/n)^{1/2} \leq \sqrt{3}$ and on
   the other hand, we have $\sqrt{2(1+\alpha)}(n+1)^{1/2} \geq
   \sqrt{2}(1+1)^{1/2}=2$ so \eqref{5.17} holds. This being given,
   it follows from our discussions above that in order for
   \eqref{5.16} to hold, it suffices to multiply both sides of
   \eqref{5.16} by $\Lambda_n$ and in the resulting expression replace the values of $\Lambda_n/\Lambda_{n+1}$ and $\Lambda_n$ by the values given
   by the right-hand side expressions of \eqref{5.12} and \eqref{5.13}
   respectively. Then after some simplifications and on setting
   $x=1/n$, we see that it suffices to show for $0 \leq x \leq
   1/2$,
\begin{equation*}
    2-2(1+x)^{1-\alpha}(1+2x)^{(\alpha-1)/2}+\frac
    {0.94(1+\alpha)x^3}{(1+x)^2}\Big ((1+2x)^{(1+\alpha)/2}-1 \Big )^{-1}
    \geq 0.
\end{equation*}
   By the mean value theorem again, we see that
   $(1+2x)^{(1+\alpha)/2}-1 \leq (1+\alpha)x$. Replacing this in
   the above inequality, we see that it suffices to show
   $h_{\alpha}(x^2/(1+x)^2) \geq 0$, where
\begin{equation*}
    h_{\alpha}(t) =2-2(1-t)^{(\alpha-1)/2}+0.94t.
\end{equation*}
    As $h_{\alpha}(t)$ is a concave function of $t$, and note that $x^2/(1+x)^2 \leq 1/9$, in order for $h_{\alpha}(x^2/(1+x)^2) \geq 0$,
    it suffices to check $h_{\alpha}(0) \geq 0$ and $h_{\alpha}(1/9) \geq 0$. This leads to the condition $\alpha \geq 1-2\ln (1+0.94/18)/\ln (9/8) < 0.14$.
    This now completes the proof.
\end{proof}

    Now we are ready to prove the following
\begin{theorem}
\label{thm5}
   For $0.14 \leq \alpha \leq 1 $ and $p \geq 2$, the sequence defined in \eqref{5.2} for $\Lambda_n=\sum^n_{i=1}i^{\alpha}$ is
   increasing.
\end{theorem}
\begin{proof}
   By Lemma 7 of \cite{B2} (with
$x_n=\Lambda^{-p}_{n}$ there), it suffices to prove inequality of
   \eqref{4.7} for $p=2$, which is
\begin{equation*}
  1+n\Big (1+\frac {(n+1)^{\alpha}}{\Lambda_n} \Big )^2-(n+1)\Big (1+\frac {(n+2)^{\alpha}}{\Lambda_{n+1}} \Big
  )^2 \geq 0.
\end{equation*}
   Expanding the squares, we can recast the above inequality as
\begin{equation*}
  \frac {2n(n+1)^{\alpha}}{\Lambda_n}-\frac
   {2(n+1)(n+2)^{\alpha}}{\Lambda_{n+1}}+\frac {n(n+1)^{2\alpha}}{\Lambda^2_n}-\frac
   {(n+1)(n+2)^{2\alpha}}{\Lambda^2_{n+1}} \geq 0.
\end{equation*}
   The assertion of the theorem now follows by combining Lemma
   \ref{lem7} and Lemma \ref{lem8}.
\end{proof}

%%----------------------------------------------------------------------------------------------------------------------
%%----------------------------------------------------------------------------------------------------------------------
   We will now show the sequence defined in \eqref{5.2} for
$\Lambda_n=\sum^n_{i=1}i^{\alpha}$ is increasing for all $0 \leq
\alpha \leq 1$, provided $p$ is large enough. We first need two
lemmas:
\begin{lemma}
\label{lem9}
    For $n \geq 1, 0 \leq \alpha
    \leq 1$ and $p \geq 1$, the function
\begin{equation*}
   f_n(x)=1+n\Big(1+\frac {(n+1)^{\alpha}}{x}\Big )^p-(n+1)\Big(1+\frac
   {(n+2)^{\alpha}}{(n+1)^{\alpha}+x}\Big )^p
\end{equation*}
    is a decreasing function for $x \leq \frac
    {n^{(1+\alpha)/2}(n+1)^{\alpha}}{(n+2)^{(1+\alpha)/2}-n^{(1+\alpha)/2}}$.
\end{lemma}
\begin{proof}
    We have
\begin{equation*}
   f'_n(x)=p(n+1)\Big(1+\frac
   {(n+2)^{\alpha}}{(n+1)^{\alpha}+x}\Big )^{p-1}\frac
   {(n+2)^{\alpha}}{((n+1)^{\alpha}+x)^2}-pn\Big(1+\frac {(n+1)^{\alpha}}{x}\Big )^{p-1}\frac
   {(n+1)^{\alpha}}{x^2}.
\end{equation*}
    To show $f'_n(x) \leq 0$, it suffices to show the following
    inequalities:
\begin{eqnarray*}
   1+\frac
   {(n+2)^{\alpha}}{(n+1)^{\alpha}+x} & \leq & 1+\frac
   {(n+1)^{\alpha}}{x}, \\
  \frac
   {(n+1)(n+2)^{\alpha}}{((n+1)^{\alpha}+x)^2} & \leq & \frac
   {n(n+1)^{\alpha}}{x^2}.
\end{eqnarray*}
    It's also easy to see that one only needs to show the above
    inequalities for $x = \frac
    {n^{(1+\alpha)/2}(n+1)^{\alpha}}{(n+2)^{(1+\alpha)/2}-n^{(1+\alpha)/2}}$,
    in which case both inequalities are easy to prove and this
    completes the proof.
\end{proof}

\begin{lemma}
\label{lem10}
    For $n \geq 1, 0 \leq \alpha
    \leq 1$, we have
\begin{equation*}
   (n+1)^{\alpha}+\frac
    {n^{(1+\alpha)/2}(n+1)^{\alpha}}{(n+2)^{(1+\alpha)/2}-n^{(1+\alpha)/2}}
    \geq \frac
    {(n+1)^{(1+\alpha)/2}(n+2)^{\alpha}}{(n+3+1/n^2)^{(1+\alpha)/2}-(n+1)^{(1+\alpha)/2}}.
\end{equation*}
\end{lemma}
\begin{proof}
    Let $x=1/n$, it is easy to see that we can recast the above
    inequality as $f(\alpha;x) \geq 0$ for $x=1/n$, where
\begin{equation*}
    f(\alpha;x)=(1+3x+x^3)^{(1+\alpha)/2}(1+x)^{(\alpha-1)/2}(1+2x)^{(1-\alpha)/2}-(1+x)^{\alpha}(1+2x)^{(1-\alpha)/2}-(1+2x)^{(1+\alpha)/2}+1.
\end{equation*}
     We regard $f(\alpha;x)$ as a function of $\alpha$ and note
     that
\begin{eqnarray*}
   &&(1+2x)^{-(1+\alpha)/2}f'(\alpha;x)  \\
   &=& \frac {1}{2}\ln \Big ( \frac {(1+3x+x^3)(1+x)}{(1+2x)}
   \Big) \cdot \Big (\frac {1+3x+x^3}{1+x} \Big )^{1/2} \cdot \Big ( \frac {(1+3x+x^3)(1+x)}{(1+2x)^2} \Big
   )^{\alpha/2} \\
   &&- \ln \Big ( \frac {(1+x)}{(1+2x)^{1/2}}
   \Big)\cdot \Big ( \frac {1+x}{1+2x} \Big
   )^{\alpha}-\ln (1+2x)^{1/2} \\
   & \geq & \frac {1}{2}\ln \Big ( \frac {(1+3x+x^3)(1+x)}{(1+2x)}
   \Big) \cdot \Big (\frac {1+3x+x^3}{1+x} \Big )^{1/2} \cdot \Big ( \frac {(1+3x+x^3)(1+x)}{(1+2x)^2} \Big
   )^{\alpha/2} \\
   &&- \ln \Big ( \frac {(1+x)}{(1+2x)^{1/2}}
   \Big)-\ln (1+2x)^{1/2}.
\end{eqnarray*}
    It's easy to see that when $0 \leq x \leq 1/2$, we have
    $(1+3x+x^3)(1+x) \leq (1+2x)^2$ and one verifies directly that
    when $x=1$, the last expression above is $\geq 0$ for either
    $\alpha=0,1$. Therefore, in order to show $f'(\alpha;x) \geq
    0$ for $x=1/n$, it suffices to assume $0 \leq x \leq 1/2$ and
    assume $\alpha=1$ in the last expression above. Therefore, it
     rests to show $h(x) \geq 0$ for $0 \leq x \leq 1/2$, where
\begin{equation*}
    h(x)=\frac {1}{2}\ln \Big ( \frac {(1+3x+x^3)(1+x)}{(1+2x)}
   \Big) \cdot \frac {1+3x+x^3}{1+2x} - \ln (1+x).
\end{equation*}
    Direction calculation shows that
\begin{equation*}
    \frac {2(1+2x)^2}{1+3x^2+4x^3}h'(x)=\frac
    {x(-2+x+8x^2+6x^3)}{(1+x)(1+3x^2+4x^3)}+\ln \Big ( \frac {(1+3x+x^3)(1+x)}{(1+2x)}
   \Big),
\end{equation*}
    and the derivative of the last expression above equals
    $x^2h_1(x)/((1+x)(1+2x)(1+3x+x^3)(1+3x^2+4x^3)^2)$, where
\begin{equation*}
     h_1(x)=96x^8+292x^7+436x^6+592x^5+610x^4+603x^3+511x^2+258x+56
     \geq 0.
\end{equation*}
     As it is easy to check $h'(0)=h(0)=0$, this now implies $h(x) \geq 0$ for
     $0 \leq x \leq 1/2$ and it follows that $f(\alpha;x)$ is an
     increasing function of $\alpha$ for $x=1/n$. In order to
     completes the proof, it remains to show $f(0;x) \geq 0$ and
     we recast this as
\begin{equation*}
     (1+3x+x^3)^{1/2}(1+x)^{-1/2}+(1+2x)^{-1/2} \geq 2.
\end{equation*}
    The above inequality can be verified by taking squares and
    this completes the proof.
\end{proof}

     Now we are ready to prove the following
\begin{theorem}
\label{thm6}
   For $0 \leq \alpha \leq 1 $ and $p \geq 8/(1+\alpha)$, the sequence defined in \eqref{5.2} for $\Lambda_n=\sum^n_{i=1}i^{\alpha}$ is
   increasing.
\end{theorem}
\begin{proof}
   Let $\Lambda_n=\sum^n_{i=1}i^{\alpha}$ and it suffices to show inequality
    \eqref{4.7} for $p \geq 8/(1+\alpha)$. Note that in our case
    we can recast inequality \eqref{4.7} as $f_n(\Lambda_n) \geq
    0$ where $f_n(x)$ is defined as in Lemma \ref{lem9}. It
    follows from the reversed inequality of \eqref{5.13} (note that it holds when $0 \leq \alpha \leq 1$) and Lemma \ref{lem9} that it suffices to show $f_n(\frac
    {n^{(1+\alpha)/2}(n+1)^{\alpha}}{(n+2)^{(1+\alpha)/2}-n^{(1+\alpha)/2}}) \geq
    0$. Equivalently, this is
\begin{equation*}
    1+n(\frac {n+2}{n})^{p(1+\alpha)/2}-(n+1)\Big ( 1+\frac
    {(n+2)^{\alpha}}{(n+1)^{\alpha}+\frac
    {n^{(1+\alpha)/2}(n+1)^{\alpha}}{(n+2)^{(1+\alpha)/2}-n^{(1+\alpha)/2}}}
    \Big )^p \geq 0.
\end{equation*}
    We now apply Lemma \ref{lem10} to see that it suffices to show
\begin{equation*}
    1+n(\frac {n+2}{n})^{p(1+\alpha)/2}-(n+1)( \frac
    {n+3+1/n^2}{n+1})^{p(1+\alpha)/2} \geq 0.
\end{equation*}
     As $p \geq 8/(1+\alpha)$, it suffices to prove the above inequality with $p(1+\alpha)/2$ replaced by $4$. In this case,
     on setting $x=1/n$, we can recast the above inequality as
\begin{equation*}
   (1+x)^3(x+(1+2x)^4)-(1+3x+x^3)^4=x^3(20+76x+60x^2-34x^3-20x^4-54x^5-4x^6-12x^7-x^9)\geq 0.
\end{equation*}
     This now completes the proof.
\end{proof}

%%------------------------------------------------------------------------------------------------------------------
%%------------------------------------------------------------------------------------------------------------------

    It follows readily from Theorem \ref{thm1}, Theorem
    \ref{thm5} and Theorem \ref{thm6} that we have the following
\begin{cor}
\label{cor5.3}
  Let ${\bf x}$ be a non-negative non-increasing sequence, then for $p \geq 2$,
  $0.14 \leq \alpha \leq 1$, or for $0 \leq \alpha \leq 1$, $p \geq 8/(1+\alpha)$, we have
\begin{equation*}
  \sum^{\infty}_{j=1}\Big(\sum^j_{k=1}\frac {k^{\alpha}}{\sum^j_{i=1}i^{\alpha}}x_k \Big )^p
  \geq \sum^{\infty}_{j=1}\Big( \frac {1}{\sum^j_{i=1}i^{\alpha}} \Big )^p ||{\bf
  x}||^p_p.
\end{equation*}
   The constant is best possible.
\end{cor}

%%-------------------------------------------------------------------------------------
\section{Applications of Theorem \ref{thm1} to N\"orlund matrices}
\label{sec 6} \setcounter{equation}{0}
%%-------------------------------------------------------------------------------------
It is asked in \cite{R&S} to determine the lower bounds for
N\"orlund matrices and motivated by this, we apply a similar idea
to that used in the proof of Theorem 4 in \cite{B} to prove the
following
\begin{lemma}
\label{lem2}
   Let ${\bf x}$ be a non-negative non-increasing sequence. Let $p \geq 1$ and let $A$ be an infinite N\"orlund matrix
   generated by $(\lambda_{j})$ with $\lambda_1 >0$. Suppose that $\Lambda_{j}/\Lambda_{j+1}$ is increasing for $j \geq
1$ and for any integer $k \geq 1, r \geq 1$,
$\Lambda_k/\Lambda_{k+1} \geq \Lambda_{kr}/\Lambda_{(k+1)r}$. Then
$||A{\bf x}||_p \geq \lambda ||{\bf x}||_p$ with the best possible
constant (provided that the infinite sum converges)
\begin{equation*}
   \lambda^p = 1+\sum^{\infty}_{j=2}(1-\frac
   {\Lambda_{j-1}}{\Lambda_j})^p.
\end{equation*}
\end{lemma}
\begin{proof}
   Theorem \ref{thm1}
implies that $||A{\bf x}||_p \geq \lambda ||{\bf x}||_p$ with
\begin{eqnarray*}
   \lambda^p &=& \inf_{r}
   r^{-1}\sum^{\infty}_{j=1}(\sum^{\min (r,j)}_{k=1}\frac {\lambda_{j-k+1}}{\Lambda_j})^p=1+\inf_{r}
   r^{-1}\sum^{\infty}_{j=r+1}(\sum^r_{k=1}\frac {\lambda_{j-k+1}}{\Lambda_j})^p \\
    &=& 1+\inf_{r}
   r^{-1}\sum^{\infty}_{j=r+1}(1-\frac {\Lambda_{j-r}}{\Lambda_j})^p= 1+\inf_{r}
   \sum^{\infty}_{k=1}a_k(r),
\end{eqnarray*}
   where
\begin{equation*}
   a_k(r)=r^{-1}\sum^{(k+1)r}_{j=kr+1}(1-\frac
   {\Lambda_{j-r}}{\Lambda_j})^p.
\end{equation*}
   It therefore remains to show that $a_k(r) \geq a_k(1)$. To show this, it suffices to show that for $k \geq 1, r \geq 1,
   kr+1 \leq j \leq (k+1)r$, we have
\begin{equation*}
   1-\frac {\Lambda_{j-r}}{\Lambda_j} \geq 1-\frac
   {\Lambda_{k}}{\Lambda_{k+1}}.
\end{equation*}
    The assumption $\Lambda_{j}/\Lambda_{j+1}$ is increasing for $j \geq
1$ implies that
\begin{equation*}
   1-\frac {\Lambda_{j-r}}{\Lambda_j} \geq 1-\frac
   {\Lambda_{(k+1)r-r}}{\Lambda_{(k+1)r}}.
\end{equation*}
   This combines with the other assumption implies the
   assertion of the lemma.
\end{proof}

   If we take $\Lambda_j=j^{\alpha}, \alpha>0$ in Lemma
   \ref{lem2}, then the assumptions there are easily verified and we
   thus have
\begin{cor}
\label{cor4}
  Let ${\bf x}$ be a non-negative non-increasing sequence, $p > 1$,
  $\alpha >0$, then
\begin{equation*}
  \sum^{\infty}_{j=1}\Big(\sum^j_{k=1}\frac {(j-k+1)^{\alpha}-(j-k)^{\alpha}}{j^{\alpha}}x_k \Big )^p
  \geq \sum^{\infty}_{j=1}\Big(\frac {j^{\alpha}-(j-1)^{\alpha}}{j^{\alpha}} \Big )^p ||{\bf
  x}||^p_p.
\end{equation*}
   The constant is best possible.
\end{cor}

   We remark here that when the assumptions of Lemma \ref{lem2}
   are satisfied by some sequence $(\Lambda_n)$, then the same
   assumptions are also satisfied by the sequence
   $(\sum^n_{i=1}\Lambda_i)$. To see this, we let $\Lambda'_n=\sum^n_{i=1}\Lambda_n$ and
   note that the fact $\Lambda'_{n}/\Lambda'_{n+1}$ is increasing
   follows from \cite[Lemma 3.1]{G7}.
   To show $\Lambda'_k/\Lambda'_{k+1} \geq
   \Lambda'_{kr}/\Lambda'_{(k+1)r}$, we apply \cite[Lemma 3.1]{G7}
   again to see that it suffices to show for $r \geq 1, n \geq 0$,
\begin{equation*}
   \frac {\Lambda_{n+1}}{\Lambda_{n+2}} \geq \frac
   {\sum^{r(n+1)}_{i=rn+1}\Lambda_i}{\sum^{r(n+2)}_{i=r(n+1)+1}\Lambda_i}.
\end{equation*}
   The above inequality holds since by our assumptions for $(\Lambda_n)$, we have for $1 \leq i \leq r$,
   $\Lambda_{rn+i}/\Lambda_{r(n+1)+i}$ $ \leq \Lambda_{rn+r}/\Lambda_{r(n+1)+r} \leq \Lambda_{n+1}/\Lambda_{n+2}$.

   We now take $\Lambda'_j=\sum^j_{i=1}i^{\alpha}, \alpha \geq 0$ so that by our remark above and Corollary \ref{cor4}, we have
\begin{cor}
\label{cor5}
  Let ${\bf x}$ be a non-negative non-increasing sequence, $p > 1$,
  $\alpha \geq 0$, then
\begin{equation*}
  \sum^{\infty}_{j=1}\Big(\sum^j_{k=1}\frac {(j-k+1)^{\alpha}}{\sum^j_{i=1}i^{\alpha}}x_k \Big )^p
  \geq \sum^{\infty}_{j=1}\Big(\frac {j^{\alpha}}{\sum^j_{i=1}i^{\alpha}} \Big )^p ||{\bf
  x}||^p_p.
\end{equation*}
   The constant is best possible.
\end{cor}

%%---------------------------------------------------------------------------------------
\vskip0.1in
%%---------------------------------------------------------------------------------------
\noindent {\bf Acknowledgements.}
  The author is supported by a research fellowship from an Academic
Research Fund Tier 1 grant at Nanyang Technological University for
this work.
%%-----------------------------------------------------------------------------------------


\begin{thebibliography}{99}
%%-----------------------------------------------------------------------------------------
%%-----------------------------------------------------------------------------------------
\bibitem{alz1.5} H. Alzer, Sharp bounds for the ratio of $q$-gamma
functions, {\em Math. Nachr.}, {\bf 222} (2001), 5--14.
%%----------------------------------------------------------------------------------------
\bibitem{BPS} S. Barza, L.-E. Persson and J. Soria, Sharp weighted multidimensional integral inequalities for monotone functions,
{\em  Math. Nachr.}, {\bf 210} (2000),  43--58.
%%----------------------------------------------------------------------------------------
\bibitem{B} G. Bennett, Lower bounds for matrices, {\em Linear Algebra Appl.}, {\bf 82} (1986),  81--98.
%%----------------------------------------------------------------------------------------
\bibitem{B1} G. Bennett, Some elementary inequalities. II., {\em  Quart. J. Math. Oxford Ser. (2)}, {\bf 39} (1988),  385--400.
%%----------------------------------------------------------------------------------------
\bibitem{B2} G. Bennett, Lower bounds for matrices. II., {\em  Canad. J. Math.}, {\bf 44} (1992),  54--74.
%%-----------------------------------------------------------------------------------------
%%\bibitem{Be1} G. Bennett, Sums of powers and the meaning of $l\sp p$, {\em
%%Houston J. Math.}, {\bf 32} (2006), 801-831.
%%----------------------------------------------------------------------------------------
\bibitem{BGE} G. Bennett and K.-G. Grosse-Erdmann, Weighted Hardy inequalities for decreasing sequences and functions,
{\em  Math. Ann.}, {\bf 334} (2006),  489--531.
%%----------------------------------------------------------------------------------------
\bibitem{B&J} G. Bennett and G. Jameson, Monotonic averages of convex functions,
{\em   J. Math. Anal. Appl.}, {\bf 252} (2000),  410--430.
%%----------------------------------------------------------------------------------------
\bibitem{Bergh} J. Bergh, Hardy's inequality---a complement, {\em  Math. Z.}, {\bf 202} (1989),  147--149.
%%----------------------------------------------------------------------------------------
\bibitem{BBP} J. Bergh, V. I. Burenkov and L.-E. Persson, Best constants in reversed Hardy's inequalities for quasimonotone functions,
{\em  Acta Sci. Math. (Szeged)}, {\bf 59} (1994),  221--239.
%%----------------------------------------------------------------------------------------
\bibitem{BBP1} J. Bergh, V. I. Burenkov and L.-E. Persson, On some sharp reversed H\"older and Hardy type inequalities,
{\em   Math. Nachr.}, {\bf 169} (1994),  19--29.
%%----------------------------------------------------------------------------------------
\bibitem{Bu} V. I. Burenkov, On the exact constant in the Hardy inequality with $0<p<1$ for monotone functions,
{\em  Proc. Steklov Inst. Math.}, {\bf 194} (1993),  59--63.
%%----------------------------------------------------------------------------------------
\bibitem{CRS} M. J. Carro, J. A. Raposo and J. Soria, Recent developments in the theory of Lorentz spaces and weighted inequalities,
{\em   Mem. Amer. Math. Soc.}, {\bf 187} (2007),  128pp.
%%-------------------------------------------------------------------------------------
\bibitem{G7} P. Gao, {Sums of powers and majorization}, {\em
   J. Math. Anal. Appl.}, {\bf 340} (2008), 1241-1248.
%%-------------------------------------------------------------------------------------
\bibitem{G9} P. Gao, {On a result of Levin and Ste\v ckin}, arXiv:0902.4070.
%%-----------------------------------------------------------------------------------------
\bibitem{HLP} G. H. Hardy, J. E. Littlewood and G. P\'{o}lya, {\em
Inequalities}, Cambridge Univ. Press, 1952.
%%----------------------------------------------------------------------------------------
\bibitem{Lai} S.-Z. Lai, Weighted norm inequalities for general operators on monotone functions, {\em
 Trans. Amer. Math. Soc.}, {\bf 340} (1993), 811-836.
%%----------------------------------------------------------------------------------------
\bibitem{L&S} V. I. Levin and S. B. Ste\v ckin, Inequalities, {\em
Amer. Math. Soc. Transl. (2)}, {\bf 14} (1960),  1--29.
%%----------------------------------------------------------------------------------------
\bibitem{L} R. Lyons, A lower bound on the Ces\`aro operator, {\em
 Proc. Amer. Math. Soc.}, {\bf 86} (1982), 694.
%%----------------------------------------------------------------------------------------
\bibitem{N} C. J. Neugebauer, Weighted norm inequalities for averaging operators of monotone functions, {\em
 Publ. Mat.}, {\bf 35} (1991), 429-447.
%%----------------------------------------------------------------------------------------
\bibitem{R&S} B. E. Rhoades and P. Sen, Lower bounds for some factorable matrices, {\em
  Int. J. Math. Math. Sci.}, {\bf 2006} (2006), Art. ID 76135, 13 pp.
%%----------------------------------------------------------------------------------------
\bibitem{S} V. Stepanov, The weighted Hardy's inequality for nonincreasing functions, {\em
  Trans. Amer. Math. Soc.}, {\bf 338} (1993), 173-186.
%%-----------------------------------------------------------------------------------------
\end{thebibliography}
\end{document}